\documentclass[leqno,12pt]{amsart} 
\usepackage{amssymb}
\theoremstyle{plain} 
\newtheorem{theorem}{\indent\sc Theorem}[section]

\newtheorem{corollary}[theorem]{\indent\sc Corollary}
\newtheorem{proposition}[theorem]{\indent\sc Proposition}

\newtheorem{fact}[theorem]{\indent\sc Fact}
\theoremstyle{definition} 

\newtheorem{remark}[theorem]{\indent\sc Remark}

\def\R{{\mathbf{R}}}
\def\H{{\mathbf{H}}}
\def\Pi{{\mathbf{P}}}

\begin{document}

\title[Heinz-type estimates in Lorentz-Minkowski space]{Heinz-type mean curvature estimates \\ in Lorentz-Minkowski space} 

\author[A.~Honda]{Atsufumi Honda} 

\author[Y.~Kawakami]{Yu Kawakami} 

\author[M.~Koiso]{Miyuki Koiso} 

\author[S.~Tori]{Syunsuke Tori} 

\dedicatory{Dedicated to Professor Atsushi Kasue on the occasion of his 65th birthday}

\renewcommand{\thefootnote}{\fnsymbol{footnote}}
\footnote[0]{2020\textit{ Mathematics Subject Classification}.
Primary 53A10; Secondary 53B30, 53C24, 53C42.}
%
%
\thanks{ 
The authors were partially supported by JSPS KAKENHI Grant Number JP18H04487, JP19K03463, JP19K14526, JP20H04642 and JP20H01801.} 
\address{
Department of Applied Mathematics, \endgraf
Faculty of Engineering, \endgraf 
Yokohama National University, \endgraf
79-5, Tokiwadai, Hodogaya-ku, Yokohama, 240-8501, Japan 
}
\email{honda-atsufumi-kp@ynu.ac.jp}

\address{
Faculty of Mathematics and Physics, \endgraf
Institute of Science and Engineering, \endgraf 
Kanazawa University, \endgraf
Kanazawa, 920-1192, Japan 
}
\email{y-kwkami@se.kanazawa-u.ac.jp}

\address{
Institute of Mathematics for Industry, \endgraf 
Kyushu University, \endgraf
744, Motooka Nishi-ku, Fukuoka, 819-0395, Japan
}
\email{koiso@math.kyushu-u.ac.jp}

\address{
Graduate School of Natural Science and Technology, \endgraf 
Kanazawa University, \endgraf
Kanazawa, 920-1192, Japan
}
\email{shunsuke-1350@stu.kanazawa-u.ac.jp}


\maketitle

\begin{abstract}
We provide a unified description of Heinz-type mean curvature estimates under an assumption on the gradient bound for space-like graphs and time-like graphs  
in the Lorentz-Minkowski space. As a corollary, we give a unified vanishing theorem of mean curvature for these entire graphs of constant mean curvature.  
\end{abstract}

\section{Introduction}\label{Sec-Int}
The celebrated theorem of Bernstein \cite{Be1915} (as for a simple proof, see \cite{Ni1957}), stating that the only entire minimal graph in the Euclidean $3$-space ${\R}^3$ is a plane, 
has played a seminal role in the evolution of surface theory. 
For instance, by using this theorem,  we can show the uniqueness theorem for an entire graph of constant mean curvature in ${\R}^3$.   
Indeed Heinz \cite{He1955} proved that if $\varphi (u^1, u^2)$ is a $C^2$-differentiable function defined on the open disk $B^2 (R)$ with center at the origin and radius $R (>0)$ in the $(u^{1}, u^{2})$-plane, and the mean curvature $H$ of the graph 
$\Gamma_{\varphi}:= \{ (u^{1}, u^{2}, \varphi (u^{1}, u^{2})) \in \R^3 \, ; \,  (u^{1}, u^{2}) \in B^2 (R) \}$ of $\varphi$ satisfies $|H|\geq \alpha >0$, where $\alpha$ is constant, then $R\leq 1/\alpha$. As a corollary, we have a vanishing theorem of mean curvature for entire graphs of constant mean curvature, that is, if  $\Gamma_{\psi}$ is an entire graph of 
constant mean curvature, then $H\equiv 0$.     
Combining these results, we obtain that the only entire graph of constant mean curvature in ${\R}^3$ 
is a plane  (see \cite[Section 2.1]{Ke2003}). The Heinz result was extended to the case of graphic hypersurfaces in the Euclidean $(n+1)$-space ${\R}^{n+1}$ 
by Chern \cite{Ch1965} and Flanders \cite{Fl1966}.  

We will study this subject for graphic hypersurfaces in the Lorentz-Minkowski space. Here we recall some basic definitions and fundamental facts. We denote by ${\R}^{n+1}_{1} = (\R^{n+1}, \langle \,  \,  ,  \, \rangle_{L})$ the Lorentz-Minkowski $(n+1)$-space with the Lorentz metric 
\[
\langle (x^{1}, \ldots , x^{n}, x^{n+1}), (y^{1}, \ldots , y^{n}, y^{n+1}) \rangle_{L} := x^{1}y^{1} + \cdots + x^{n}y^{n} -x^{n+1}y^{n+1},  
\]
where $ (x^{1}, \ldots , x^{n}, x^{n+1})$, $(y^{1}, \ldots , y^{n}, y^{n+1}) \in \R^{n+1}$. Let $M^{n}$ be a smooth orientable $n$-manifold. 
An immersion $f\colon M^n \to \R^{n+1}_{1}$ is called {\it space-like} 
if the induced metric $g:= f^{\ast} \langle \,  \,  ,  \, \rangle_{L}$ on $M^n$ is Riemannian, and is called {\it time-like} if $g$ on $M^n$ is Lorentzian.    
A space-like (resp. time-like) immersion of constant mean curvature is called a {\it space-like (resp. time-like) constant mean curvature immersion}. 
In particular, a space-like immersion with vanishing mean curvature is called a {\it space-like maximal immersion} and 
a time-like immersion with vanishing mean curvature is called a {\it time-like minimal immersion}. 

Let $\psi \colon \Omega \to \mathbf{R}$ be a $C^2$-differentiable function defined on a domain $\Omega$ in $\R^{n}$. We consider the graph $\Gamma_{\psi}:= \{ (u^{1}, \ldots, u^{n}, \psi (u^{1}, \ldots , u^{n}))\in {\R}^{n+1}_{1} \, ; \, 
(u^{1}, \ldots , u^{n})\in \Omega \}$ of $\psi$. If $\Omega = \R^{n}$, then $\Gamma_{\psi}$ is called an {\it entire graph} of $\psi$ in ${\R}^{n+1}_{1}$. 
If the graph $\Gamma_{\psi}$ of $\psi$ is space-like, the gradient $\nabla \psi := (\psi_{u^1}, \ldots, \psi_{u^n})$ of $\psi$ satisfies $|\nabla \psi|:= \sqrt{(\psi_{u^{1}})^2 +  
\cdots + (\psi_{u^{n}})^2} <1$ on $\Omega$, and 
\begin{equation}\label{1-1}
\displaystyle nH = \sum_{i=1}^{n} \dfrac{\partial }{\partial u^i}\left( \dfrac{\psi_{u^i}}{\sqrt{1-|\nabla \psi|^2}} \right) = \text{div} \left( \dfrac{\nabla \psi}{\sqrt{1-|\nabla \psi|^2}}\right) , 
\end{equation}
that is, 
\[
\displaystyle (1-|\nabla \psi|^{2})\sum_{i=1}^{n} \dfrac{\partial^{2} \psi}{(\partial u^{i})^{2}} +\sum_{i, j=1}^{n}\dfrac{\partial \psi}{\partial u^{i}} \dfrac{\partial \psi}{\partial u^{j}} \dfrac{\partial^{2} \psi}{\partial u^{i} \partial u^{j}}= nH(1-|\nabla \psi|^{2})^{3/2}  
\]
holds. Here $H$ is the mean curvature of $\Gamma_{\psi}$ and $\psi_{u^i}:= \partial \psi / \partial u^{i}$. 
If the graph $\Gamma_{\psi}$ of $\psi$ is time-like, its gradient $\nabla \psi$ satisfies $|\nabla \psi |>1$ on $\Omega$, and 
\begin{equation}\label{1-1-2}
\displaystyle nH = \sum_{i=1}^{n} \dfrac{\partial }{\partial u^i}\left( \dfrac{\psi_{u^i}}{\sqrt{|\nabla \psi|^2 -1}} \right) = \text{div} \left( \dfrac{\nabla \psi}{\sqrt{|\nabla \psi|^2 -1}}\right) 
\end{equation}
holds. 

For entire space-like maximal graphs in $\R^{n+1}_{1}$, the following uniqueness theorem, called the Calabi-Bernstein theorem, is well-known.   
This result was proved by Calabi \cite{Ca1968} for $n\leq 4$ and Cheng-Yau \cite{CY1976} for all $n$.  
\begin{fact}[The Calabi-Bernstein theorem]\label{fact1}
Any entire space-like maximal graph is a space-like hyperplane in  $\R^{n+1}_{1}$. 
\end{fact}
We remark that an improvement of this theorem and a fluid mechanical interpretation of the duality between minimal graphs in ${\R}^3$ and space-like maximal graphs in $\R^{3}_{1}$ are discussed in \cite{AUY2019} (cf. \cite{Be1958}).  

On the other hand, for entire space-like constant mean curvature graphs, the uniqueness theorem does not hold in general. In fact, the graph of the function  
\begin{equation}\label{eq-hyperboloid}
\psi (u^1, \ldots, u^n) = \left( (u^1)^2 + \cdots + (u^n)^2 + \dfrac{1}{H^2} \right)^{1/2} \quad (H>0) 
\end{equation}
is an entire space-like constant mean curvature $H$ graph which is not a hyperplane. This graph is a hyperboloid. 
Many other entire space-like constant mean curvature graphs are constructed by Treibergs \cite{Tr1982}.   
Hence the following questions naturally arise. What are the situations where Heinz-type mean curvature estimate and Bernstein-type uniqueness theorem hold for space-like graphs in ${\R}^{n+1}_{1}$?  What are these situations for time-like graphs in ${\R}^{n+1}_{1}$? 

In this paper, we perform a systematic study for these questions. In particular, we give a unified description of Heinz-type mean curvature estimates under an assumption on 
the gradient bound for space-like graphs and time-like graphs 
in ${\R}^{n+1}_{1}$ (Theorem \ref{thm1}). As a corollary, we obtain a unified vanishing theorem of mean curvature for entire graphs of constant mean curvature 
with a gradient bound (Corollary \ref{cor1}). In Section \ref{Sec-App}, by applying these results, we answer the questions mentioned above.  
Moreover we give a short commentary about the recent study by the authors for entire constant mean curvature graphs in $\R^{n+1}_{1}$. 

Finally, the authors gratefully acknowledge the useful comments from Shintaro Akamine and Atsushi Kasue during the preparation of this paper. 
 

\section{Main results}\label{Sec-Main}
The main theorem is stated as follows: 
\begin{theorem}\label{thm1}
Let $B^{n}(R)$ be the open $n$-ball with center at the origin and radius $R (> 0)$ in $\mathbf{R}^{n}$, 
and $\psi (u^{1}, \ldots , u^{n})$ a $C^2$-differentiable function on  $B^{n}(R)$. Suppose that there exist constants $M>0$ and $k\in \R$ such that 
\begin{equation}\label{eq2-0}
\dfrac{|\nabla \psi|}{\sqrt{|1-|\nabla \psi|^2|}}\leq M \biggl{(}(u^1)^2+\cdots +(u^n)^2  \biggr{)}^{k} 
\end{equation}
on $B^{n}(R)$. Assume that $\Gamma_{\psi}$ is a space-like or time-like graph of $\psi$ in $\R^{n+1}_{1}$. 
If the mean curvature $H$ of $\Gamma_{\psi}$ satisfies the inequality 
\[
|H|\geq \alpha >0, 
\]
where $\alpha$ is constant, then the following inequality holds: 
\begin{equation}\label{eq2-1}
\alpha \leq MR^{2k-1}. 
\end{equation}
\end{theorem}

\begin{proof}
Let $\omega$ be the $(n-1)$-form given by 
\begin{equation}\label{eq2-a}
\displaystyle \omega = \sum_{i=1}^{n} (-1)^{i-1} \left( \dfrac{\psi_{u^i}}{\sqrt{|1-|\nabla \psi|^2}|} \right)\, du^{1}\wedge \cdots \wedge \widehat{du^i}\wedge \cdots \wedge du^n,  
\end{equation}
where $\widehat{du^i}$ means that $du^i$ is omitted from the wedge product. Then we can obtain $d\omega = nH\, du^{1}\wedge \cdots \wedge du^n$ 
by \eqref{1-1} and \eqref{1-1-2}.  
Take any positive number $R'$ satisfying $0< R' <R$. The Stokes theorem yields the following equality:
\begin{equation}\label{eq2-2}
\displaystyle \int_{\overline{B}^{n}(R')} d\omega = \int_{\partial B^{n}(R')} \omega, 
\end{equation} 
where $\overline{B}^{n}(R')$ is the closure and $\partial B^{n}(R')$ is the boundary of  $B^{n}(R')$. 

We may assume $H\geq \alpha >0$ by changing the direction of the normal vector if necessary. Let $V_{n}$ be the volume of the unit closed $n$-ball $\overline{B}^n (1)$ and 
$A_{n-1}$ the volume of the unit $(n-1)$-sphere $\partial B^n (1)$. Then, by radial integration, $nV_{n}=A_{n-1}$ holds. The absolute value of the left-hand side of \eqref{eq2-2} becomes 
\begin{eqnarray*}
\left| \displaystyle \int_{\overline{B}^{n}(R')} d\omega  \right| & =  & n\int_{\overline{B}^{n}(R')} H\, du^{1}\cdots du^{n} \\
                                                                                           & \geq & n\alpha \int_{\overline{B}^{n}(R')} du^{1}\cdots du^{n} = n\alpha (R')^n V_{n}.  \\
\end{eqnarray*}
On the other hand, by using the Cauchy-Schwarz inequality, the absolute value of the right-hand side of \eqref{eq2-2} becomes 
\begin{eqnarray*}
\left| \int_{\partial B^{n}(R')} \omega \right| &=& \left| \int_{\partial B^{n}(R')}  \sum_{i=1}^{n} (-1)^{i-1} \left( \dfrac{\psi_{u^i}}{\sqrt{|1-|\nabla \psi|^2|}} \right)\, du^{1}\cdots \widehat{du^i}\cdots du^n \right| \\
                                                                &\leq & \displaystyle \int_{\partial B^{n}(R')} \dfrac{|\nabla \psi|}{\sqrt{|1- |\nabla \psi|^2|}} \left( \sum_{i=1}^{n} \left( du^{1}\cdots \widehat{du^i}\cdots du^n \right)^{2} \right)^{1/2} . \\
\end{eqnarray*}
By \eqref{eq2-0}, we have 
\begin{eqnarray*}
\displaystyle \left| \int_{\partial B^{n}(R')} \omega \right| &\leq & M (R')^{2k} \int_{\partial B^{n}(R')} \left( \sum_{i=1}^{n} \left( du^{1}\cdots \widehat{du^i}\cdots du^n \right)^{2} \right)^{1/2} \\
                                                                                     &=& M (R')^{n+2k-1} A_{n-1}. \\
\end{eqnarray*} 
We thus obtain $n\alpha (R')^{n}V_{n} \leq M(R')^{n+2k-1}A_{n-1}$, that is, $\alpha\leq M(R')^{2k-1}$. The proof is completed by letting $R' \to R$. 
\end{proof}

From the above argument, we give the following generalization of the result which was obtained by Salavessa \cite[Theorem 1.5]{Sa2008} for space-like 
graphs in ${\R}^{n+1}_{1}$. 

\begin{proposition}\label{thm-sala}
Let $D$ be a relatively compact domain (i.e. its closure $\overline{D}$ is compact) of ${\R}^n$ with smooth boundary $\partial D$.  
Assume that $\Gamma_{\psi} \subset {\R}^{n+1}_{1}$ is a space-like or time-like graph of  a $C^2$-differentiable function $\psi$. 
Set $m_{D}:= \max_{\overline{D}} |\nabla \psi|$ if $\Gamma_{\psi}$ is a space-like graph in $\R^{n+1}_{1}$ 
and $m_{D}:= \min_{\overline{D}} |\nabla \psi|$ if $\Gamma_{\psi}$ is a time-like graph 
in $\R^{n+1}_{1}$. Then, for the mean curvature $H$ of $\Gamma_{\psi}$, we have 
\begin{equation}\label{eq-sala-1}
\displaystyle \min_{\overline{D}} |H| \leq \dfrac{1}{n} \dfrac{m_{D}}{\sqrt{|1- (m_{D})^2|}} \dfrac{A (\partial D)}{V (\overline D)}. 
\end{equation} 
Here $V(\overline{D})$ (resp. $A (\partial D)$\,) is the volume of $\overline{D}$ (resp. $\partial D$). 
\end{proposition} 

\begin{remark}
To be precise, in \cite[Theorem 1.5]{Sa2008} Salavessa obtained this inequality for space-like graphs of $M\times N$, where $M$ 
is a Riemannian $n$-manifold and $N$ is an oriented Lorentzian $1$-manifold.    
\end{remark}

Indeed, from the Stokes theorem for the $(n-1)$-form $\omega$ given by \eqref{eq2-a}, we obtain 
\begin{equation}\label{eq2-2-2}
\displaystyle \int_{\overline{D}} d\omega = \int_{\partial D} \omega. 
\end{equation}
Then the absolute value of the left-hand side of \eqref{eq2-2-2} becomes 
\[
\left| \displaystyle \int_{\overline{D}}\, d\omega \,  \right| \geq n \, \min_{\overline{D}} |H|\, V (\overline{D}). 
\]
Since the function $t/\sqrt{1-t^2}$ is monotone increasing on $0< t< 1$ and $t/\sqrt{t^2 -1}$ is monotone decreasing on $t> 1$, we have 
\[
\left| \int_{\partial D} \, \omega \, \right| \leq \dfrac{m_{D}}{\sqrt{|1- (m_{D})^2|}}\, A (\partial D). 
\]
We have thus proved the inequality \eqref{eq-sala-1}. 
 
As a corollary of Theorem \ref{thm1}, we give the following unified vanishing theorem of mean curvature for entire space-like graphs and time-like graphs 
of constant mean curvature in ${\R}^{n+1}_{1}$. 

\begin{corollary}\label{cor1} 
Assume that $\Gamma_{\psi}$ is an entire space-like or time-like graph of a $C^2$-differentiable function $\psi (u^{1}, \ldots, u^{n})$ in $\R^{n+1}_{1}$. 
If the entire graph $\Gamma_{\psi}$ has constant mean curvature and there exist constants $M>0$ and $\varepsilon >0$ such that 
\begin{equation}\label{eq-Be}
\dfrac{|\nabla \psi|}{\sqrt{|1-|\nabla \psi|^2|}}\leq M \biggl{(}(u^1)^2+\cdots +(u^n)^2  \biggr{)}^{(1/2) - \varepsilon}
\end{equation}
on ${\R}^n$, then its mean curvature must vanish everywhere.  
\end{corollary}
\begin{proof}
By Theorem \ref{thm1}, the mean curvature $H$ of ${\Gamma}_{\psi}$ satisfies 
\[
|H|\leq \dfrac{M}{R^{2\varepsilon}} 
\]
on $B^{n}(R)$. We obtain $H\equiv 0$ by letting $R\rightarrow +\infty$. 
\end{proof}

\section{Applications}\label{Sec-App}
\subsection{Space-like case}\label{subsec-sp}
Then $|\nabla \psi|< 1$ holds. We give a geometric interpretation for $|\nabla \psi|/ \sqrt{1-|\nabla \psi|^2}$. 
When the graph $\Gamma_{\psi}$ is space-like, 
\begin{equation}\label{eq2-nu}
\nu (=  \nu (u^1, \cdots, u^n)) =\dfrac{1}{\sqrt{1- |\nabla \psi|^2}} (\psi_{u^1}, \cdots, \psi_{u^n}, 1)
\end{equation}
is the time-like unit normal vector field $\nu$ of $\Gamma_{\psi}$. Since $e_{n+1}=(0, \cdots , 0, 1) \in \R^{n+1}_{1}$ is also time-like, 
there exists a unique real-valued function $\theta (= \theta (u^1, \cdots , u^n ) )\geq 0$ such that $\langle \nu , e_{n+1} \rangle_{L} = -\cosh{\theta}$. 
This function $\theta$ is called the {\it hyperbolic angle} between $\nu$ and $e_{n+1}$ (see \cite{On1983}). By simple calculation, we have 
\begin{equation}\label{eq2-sinh}
\sinh{\theta} = \dfrac{|\nabla \psi|}{\sqrt{1-|\nabla \psi|^2}}. 
\end{equation}
We remark that L\'opez \cite{Lo2020} studies the function $|\nabla \psi|/ \sqrt{1-|\nabla \psi|^2}$ from the viewpoint of the 
Dirichlet problem. 

From Theorem \ref{thm1}, we obtain the following Heinz-type mean curvature estimate.  

\begin{corollary}\label{corA}
Let $\Gamma_{\psi}$ be a space-like graph of  a $C^2$-differentiable function $\psi (u^{1}, \ldots, u^{n})$ on $B^{n}(R)$ in ${\R}^{n+1}_{1}$ and 
$\theta$ the hyperbolic angle between $\nu$ and $e_{n+1}$. Suppose that there exist $M >0$ and $k\in \R$ such that 
\begin{equation}\label{eq-corA}
\sinh{\theta} \leq M \biggl{(}(u^1)^2+\cdots +(u^n)^2  \biggr{)}^{k} 
\end{equation}
on $B^{n}(R)$. If the mean curvature $H$ of $\Gamma_{\psi}$ satisfies the inequality $|H|\geq \alpha >0$, where $\alpha$ is constant, then $\alpha \leq MR^{2k-1}$ holds. 
\end{corollary}

By virtue of Fact \ref{fact1} and Corollary \ref{cor1}, we obtain the following Bernstein-type theorem for entire space-like constant mean curvature 
graphs in ${\R}^{n+1}_{1}$. 

\begin{corollary}\label{corB}
If an entire space-like graph $\Gamma_{\psi}$ of a $C^2$-differentiable function $\psi (u^{1}, \ldots, u^{n})$ in ${\R}^{n+1}_{1}$ has constant mean curvature and there exist constants $M>0$ and $\varepsilon >0$ 
such that 
\begin{equation}\label{eq-corB}
\sinh{\theta}\leq M \biggl{(}(u^1)^2+\cdots +(u^n)^2  \biggr{)}^{(1/2) - \varepsilon}
\end{equation}
on $\R^n$, then it must be a space-like hyperplane. Here $\theta$ is the hyperbolic angle between $\nu$ and $e_{n+1}$. 
\end{corollary}

Corollary \ref{corB} is optimal because there exists an example which is not congruent to a space-like hyperplane and satisfies \eqref{eq-corB} for $\varepsilon =0$. In fact, 
the function $\psi$ given by \eqref{eq-hyperboloid} is not linear and its graph has constant mean curvature. Moreover it satisfies 
\[
\sinh{\theta} = H \sqrt{(u^1)^2+\cdots +(u^n)^2},  
\] 
which is the equality in \eqref{eq-corB} for $\varepsilon =0$. 

From the case where $\varepsilon =1/2$ in Corollary \ref{corB}, we obtain the following  Liouville-type result.   

\begin{corollary}\label{cor2}   
Let $\Gamma_{\psi}$ be an entire space-like constant mean curvature graph of  a $C^2$-differentiable function $\psi (u^{1}, \ldots, u^{n})$ 
in $\R^{n+1}_{1}$ and $\theta$ the hyperbolic angle between $\nu$ and $e_{n+1}$. 
If $\sinh{\theta}$ is bounded on ${\R}^n$, then ${\Gamma}_{\psi}$ must be a space-like hyperplane.  
\end{corollary}

The following result is equivalent to Corollary \ref{cor2}. 

\begin{corollary}{\cite[Corollary 9.5]{AMR2016}}\label{cor3}
If an entire space-like graph $\Gamma_{\psi}$ of a $C^2$-differentiable function $\psi (u^{1}, \ldots, u^{n})$ in $\R^{n+1}_{1}$ has constant mean curvature and there exists some constant $C(<1)$ such that 
$|\nabla \psi|\leq C$ on ${\R}^n$, then it must be a space-like hyperplane. 
\end{corollary}

Indeed, if $|\nabla \psi | \leq C $ holds on $\R^n$, then we obtain 
\[
\dfrac{|\nabla \psi|}{\sqrt{1-|\nabla \psi|^2}}\leq \dfrac{C}{\sqrt{1-C^2}}
\]
on $\R^n$, that is, $\sinh{\theta}$ is bounded because the function $t/\sqrt{1-t^2}$ is monotone increasing on $0<t<1$. 
By the same argument, the converse also holds, that is, if $\sinh{\theta}$ is bounded on ${\R}^n$, then there exists some constant $C (<1)$ such that 
$|\nabla \psi|\leq C$ on ${\R}^n$.  

As described in \cite[Section 9.3]{AMR2016},  Corollary \ref{cor3} is a case of nonparametric form of the following  value-distribution-theoretical property 
of the Gauss map.  
\begin{fact}[\cite{Ai1992, Xi1991}, \cite{Pa1990} for a first weaker version]\label{fact2}
Let $f\colon M^n \to {\R}^{n+1}_{1}$ be a complete space-like constant mean curvature immersion. 
If the image of the time-like unit normal vector field (i.e., Gauss map) $\nu (M^{n})$ is contained in a geodesic ball in the hyperbolic $n$-space ${\H}^{n}(-1)$, 
then $f(M^{n})$ is a space-like hyperplane in ${\R}^{n+1}_{1}$. 
\end{fact}
In fact, every entire space-like constant mean curvature graph in ${\R}^{n+1}_{1}$ is complete (see \cite{CY1976}, note that the converse statement is also true, that is, 
every complete space-like hypersurface in ${\R}^{n+1}_{1}$ is an entire graph (\cite{ARS1995}))  
and the image $\nu ({\R}^n)$ is contained in a geodesic ball in ${\H}^{n}(-1)$ if and only if $\sqrt{1- |\nabla \psi|^2}$ is bounded away from zero on ${\R}^n$. However the proof of Fact \ref{fact2} is not as simple as our argument in this paper because it is obtained via an application of the Omori-Yau maximum principle (\cite{CX1992}, \cite{Om1967}, \cite{Ya1975}). See the book \cite{AMR2016} for details.   

For a space-like graph in ${\R}^{n+1}_{1}$, $\sinh{\theta}$ can be replaced with $1/\sqrt{1-|\nabla \psi|^2}$ in \eqref{eq-corA} and \eqref{eq-corB}. 
In fact, 
\[
\sinh{\theta}= \dfrac{|\nabla \psi|}{\sqrt{1-|\nabla \psi|^2}} < \dfrac{1}{\sqrt{1-|\nabla \psi|^2}} 
\]
holds. Thus we have the following result. 

\begin{corollary}\label{cor-Do}
If an entire space-like graph $\Gamma_{\psi}$ of a $C^2$-differentiable function $\psi (u^{1}, \ldots, u^{n})$ in ${\R}^{n+1}_{1}$ has constant mean curvature and there exist constants $M>0$ and $\varepsilon >0$ 
such that 
\[
\dfrac{1}{\sqrt{1-|\nabla \psi|^2}}\leq M \biggl{(}(u^1)^2+\cdots +(u^n)^2  \biggr{)}^{(1/2) - \varepsilon}
\]
on $\R^n$, then it must be a space-like hyperplane. 
\end{corollary}

From Corollary \ref{cor-Do}, we can show the following proposition which was given by Dong.   

\begin{proposition}[\cite{Do2007}]
Let $\Gamma_{\psi}$ be an entire space-like constant mean curvature graph of a $C^2$-differentiable function $\psi (u^{1}, \ldots, u^{n})$ in ${\R}^{n+1}_{1}$. 
If the function $\psi$ satisfies 
\[
\dfrac{1}{\sqrt{1-|\nabla \psi|^2}} = o (r)  \quad {as} \quad r\to +\infty, 
\]
where $r=\sqrt{(u^1)^2+\cdots +(u^n)^2}$, then $\Gamma_{\psi}$ must be a space-like hyperplane. 
\end{proposition}

\begin{remark}
In \cite{Do2007}, Dong showed a generalization of this result for an entire space-like graph of parallel  mean curvature in the pseudo-Euclidean $(m+n)$-space ${\R}^{m+n}_{n}$ of index 
$n$. 
\end{remark}

\subsection{Time-like case}\label{Subsec-time}
Then $|\nabla \psi|> 1$ holds. By Theorem \ref{thm1}, we obtain the following Heinz-type estimate for the mean curvature of 
time-like graphs in $\R^{n+1}_{1}$. 
 
\begin{corollary}\label{thm2}
Let $\Gamma_{\psi}$ be a time-like graph of  a $C^2$-differentiable function $\psi (u^{1}, \ldots , u^{n})$ on $B^{n}(R)$ in $\R^{n+1}_{1}$. 
Suppose that there exist constants $M>0$ and $k\in \R$ such that 
\begin{equation}\label{eq3-0}
\dfrac{|\nabla \psi|}{\sqrt{|\nabla \psi|^2 -1}}\leq M \biggl{(}(u^1)^2+\cdots +(u^n)^2  \biggr{)}^{k} 
\end{equation}
on $B^{n}(R)$. If the mean curvature $H$ of $\Gamma_{\psi}$ satisfies the inequality 
\[
|H|\geq \alpha >0, 
\]
where $\alpha$ is constant, then the following inequality holds: 
\begin{equation}\label{eq3-1}
\alpha \leq MR^{2k-1}. 
\end{equation}
\end{corollary}

Moreover, we obtain the following result from Corollary \ref{cor1}.   
  
\begin{corollary}\label{cor4}
If an entire time-like graph $\Gamma_{\psi}$ of a $C^2$-differentiable function $\psi (u^{1}, \ldots, u^{n})$ 
in $\R^{n+1}_{1}$ has constant mean curvature $H$ and 
there exist constants $M>0$ and $\varepsilon >0$ such that 
\begin{equation}\label{eq-Be2}
\dfrac{|\nabla \psi|}{\sqrt{|\nabla \psi|^2 -1}}\leq M \biggl{(}(u^1)^2+\cdots +(u^n)^2  \biggr{)}^{(1/2) - \varepsilon}
\end{equation}
on ${\R}^n$, then it must be minimal (i.e. $H\equiv 0$). 
\end{corollary}

We note that there exist many entire time-like minimal graphs in $\R^{n+1}_{1}$. In fact, for a $C^2$-differentiable function $h(t)$ of one variable which satisfies $h'(t) >0$ for 
all $t\in \R$, $\psi (u^1 , \cdots, u^n):= u^n +h(u^1)$ gives an entire time-like minimal graph ${\Gamma}_{\psi}$ in ${\R}^{n+1}_{1}$.  

We do not know whether Corollary \ref{cor4} is optimal or not. However, for example, from the case where $\varepsilon =1/2$ and $n=2$ in Corollary \ref{cor4}, 
we provide the following result.   

\begin{corollary}\label{cor5}
Let $\Gamma_{\psi}$ be an entire time-like constant mean curvature graph of a $C^2$-differentiable function $\psi (u^1, u^2)$ in $\R^{3}_{1}$. 
If $|\nabla \psi|/\sqrt{|\nabla \psi|^2 -1}$ is bounded on $\R^2$, then $\Gamma_{\psi}$ is a translation surface,   
that is, $\psi (u^1, u^2) = h (u^1) + k(u^2) $, where $h (u^1)$ is a $C^2$-differentiable function of $u^1$ and $k (u^2)$ is a $C^2$-differentiable function of $u^2$.  
\end{corollary}

The latter half of the statement of Corollary \ref{cor5} follows from the following fact. 

\begin{fact}[\cite{Ma1989, Mi1987}]\label{fact3}
Every entire time-like minimal graph in ${\R}^{3}_{1}$ is a translation surface.  
\end{fact}

\subsection{Closing Remark}\label{subsec-CR} 
This paper is focused on the study of entire space-like graphs and time-like graphs in ${\R}^{n+1}_{1}$.  
Here we explain the recent development of the study of entire constant mean curvature graphs in $\R^{n+1}_{1}$. 
The following definitions are given in \cite{AHUY2019}.     
Let $\psi \colon \Omega \to \R$ be a $C^2$-differentiable function defined on a domain $\Omega$ in ${\R}^n$. 
A point where $|\nabla \psi | <1$ (resp. $|\nabla \psi|> 1$, $|\nabla \psi| = 1$) is called a {\it space-like} (resp. {\it time-like}, {\it light-like}) point of 
${\Gamma}_{\psi}$. If $\Omega$ contains only space-like (resp. time-like) points, then $\Gamma_{\psi}$ is space-like (resp. time-like).  
We remark that the first fundamental form $g$ of the graph ${\Gamma}_{\psi}$ in ${\R}^{n+1}_{1}$ is degenerate at each light-like point. 
If $\Omega$ contains both space-like points and time-like points, then ${\Gamma}_{\psi}$ is called of {\it mixed type}. 
Akamine, the first author, Umehara and Yamada proved the following fact. 

\begin{fact}{\cite[Corollary C]{AHUY2019}}\label{fact-light} 
An entire real analytic constant mean curvature graph in ${\R}^{n+1}_{1}$ which has no time-like points and does have a light-like point must be a light-like hyperplane.  
\end{fact}   

See \cite{AHUY2019} (also \cite{UY2019}) for the precise definition of constant mean curvature graph which does not depend on its causal type. 
By Fact \ref{fact-light}, we know that there exists no entire real analytic constant mean curvature graph in ${\R}^{n+1}_{1}$ which has both space-like points and light-like points although there exist many nontrivial examples of entire space-like constant mean curvature graphs.  
Moreover, the first and third authors, Kokubu, Umehara and Yamada \cite[Remark 2.2]{HKKUY2017} showed that there does not exist any real analytic nonzero constant mean curvature graph of mixed type in $\R^{n+1}_{1}$.  
\begin{remark}
Several entire zero mean curvature graphs of mixed type in $\R^{3}_{1}$ were found in recent years. A systematic construction of this class is given in \cite{FKKRUYY2016}. 
See \cite{HST2020} for details about mixed type surfaces in Lorentzian 3-manifolds.  
\end{remark}



\end{document}